\documentclass[a4paper]{article}
\usepackage{amsfonts,amsmath,amstext,amssymb}
\usepackage{graphicx}

\def\rh#1{{\rm\bf H}_{\mathbb R}^{#1}}
\def\ch#1{{\rm\bf H}_{\mathbb C}^{#1}}
\def\pu#1{{\rm PU}(#1,1)}
\def\su#1{{\rm SU}(#1,1)}

\def\u#1{{\rm U}(#1,1)}
\def\isom#1{{\rm Isom}(#1)}
\def\Heis{{\mathcal H}}
\def\Siegel{{\mathcal S}}

\def\mod#1{\vert #1\vert}
\def\hmod#1{\parallel #1\parallel}
\def\hherm#1{{\langle\!\langle #1 \rangle\!\rangle}}
\def\A{\mathbb A}

\def\J{\mathbb J}
\def\R{\mathbb R}
\def\C{\mathbb C}
\def\D{\mathbb D}
\def\Z{\mathbb Z}
\def\Q{\mathbb Q}
\def\P{\mathbb P}
\def\V{\mathbb V}

\def\E{\mathcal E}

\def\herm#1{\langle #1\rangle}
\newcommand{\dps}{\displaystyle}
\def\qi{q_\infty}
\def\Qi{Q_\infty}

\def\cym#1{\rho_0(#1)}
\def\inte#1{{\rm Int}\left(#1\right)}
\def\exte#1{{\rm Ext}\left(#1\right)}

\newenvironment{proof}{\noindent\normalsize {\sc Proof}:}{{\hfill \rule{1mm}{3mm}}}

\newtheorem{theorem}{Theorem}[section]
\newtheorem{co}{Corollary}[section]
\newtheorem{prop}{Proposition}[section]
\newtheorem{lemma}{Lemma}[section]

\title{A note on trace fields of complex hyperbolic groups}

\author{Heleno Cunha\ \ \ \ \ \ \ \ \ \ Nikolay Gusevskii\thanks{Corresponding author. Supported by CNPq and FAPEMIG.}\\
cunha@mat.ufmg.br\ \ \ \ nikolay@mat.ufmg.br\\
\\
Departamento de Matem\'{a}tica\\
Universidade Federal de Minhas Gerais \\
Belo Horizonte -- MG\\
Brazil\\
30123-970}

\date{}

\begin{document}
\maketitle


\begin{abstract}
\noindent We show that if $\Gamma$ is an irreducible subgroup of
${\rm SU}(2,1)$, then $\Gamma$ contains a loxodromic element $A$. If
$A$ has eigenvalues $\lambda_1 = \lambda e^{i\varphi},$ $\lambda_2 =
e^{-2i\varphi}$, $\lambda_3 = \lambda^{-1}e^{i\varphi}$,  we prove
that $\Gamma$ is conjugate in ${\rm SU}(2,1)$ to a subgroup of ${\rm
SU}(2,1,\Q(\Gamma,\lambda)),$ where $\Q(\Gamma, \lambda)$ is the
field generated by the trace field $\Q(\Gamma)$ of $\Gamma$ and
$\lambda$. It follows from  this that if $\Gamma$ is an irreducible
subgroup of ${\rm SU}(2,1)$ such that the trace field $\Q(\Gamma)$
is real, then $\Gamma$ is conjugate in ${\rm SU}(2,1)$ to a subgroup
of ${\rm SO}(2,1)$. As a geometric application of the above, we get
that if $G$ is an irreducible discrete subgroup of ${\rm PU}(2,1)$,
then $G$ is an $\R$-Fuchsian subgroup of ${\rm PU}(2,1)$ if and only
if the invariant trace field $k(G)$ of $G$ is real.
\end{abstract}



\section*{Introduction}

Arithmetic methods are a  powerful tool in the study of Kleinian
groups, discrete subgroups of $\rm{PSL}(2,\C),$ especially of
finite-covolume discrete groups, as was demonstrated in \cite{MaR},
see also an extensive bibliography there. A central theme in this
theory is to understand the structure of the invariant trace field
and the invariant (quaternion) algebra associated to a Kleinian
group. In the case of complex hyperbolic geometry, that is, in the
case of subgroups of ${\rm PU}(n,1)$ (${\rm SU}(n,1)$) little known
about these objects, see for instance \cite{Mc}, where the study of
the invariant trace fields and the invariant algebras associated to
subgroups of ${\rm SU}(n,1)$ was initiated. In particular, in this
work the invariant trace field and the invariant algebra were
introduced for subgroups of ${\rm SU}(n,1)$. An important problem
here is to understand whether a subgroup of ${\rm SU}(n,1)$ can be
realized over the field generated by the eigenvalues of its
elements. In this paper, we prove that any irreducible subgroup
$\Gamma$ of ${\rm SU}(2,1)$ contains a loxodromic element $A$ and it
can be realized over the field generated by the trace field of
$\Gamma$ and the eigenvalues of $A$.

\vspace{2mm}

The main result of our paper is the following theorem:

\vspace{2mm}

\noindent \textbf{Theorem A} \textit{Let $\Gamma$ be an irreducible
subgroup of ${\rm SU}(2,1)$ and  $A \in \Gamma$ be loxodromic with
eigenvalues $\lambda_1 = \lambda e^{i\varphi},$  $\lambda_2 =
e^{-2i\varphi}$, $\lambda_3 = \lambda^{-1}e^{i\varphi}$. Then
$\Gamma$ is conjugate in ${\rm SU}(2,1)$ to a subgroup of ${\rm
SU}(2,1,\Q(\Gamma, \lambda))$, where $\Q(\Gamma, \lambda)$ is the
field generated by the trace field $\Q(\Gamma)$ of $\Gamma$ and
$\lambda$.}

\vspace{2mm}

As a corollary of this theorem, we get the  following

\vspace{2mm}

\noindent \textbf{Theorem B} \textit{Let $\Gamma$ be an irreducible
subgroup of ${\rm SU}(2,1)$ such that $\Q(\Gamma)$ is a subset of
$\R$, then $\Gamma$ is conjugate in ${\rm SU}(2,1)$ to a subgroup of
${\rm SO}(2,1)$.}

\vspace{2mm}

We would like to stress that in Theorem A and Theorem B we do not
assume that the group $\Gamma$ is discrete.

\vspace{2mm}

Also, in this paper, we define an invariant trace field for
subgroups of ${\rm PU}(2,1)$. Let $G$ be a subgroup of ${\rm
PU}(2,1)$ and $\Gamma = \pi^{-1}(G)$, where $\pi : {\rm SU}(2,1)
\rightarrow {\rm PU}(2,1)$ is a natural projection. Then the  {\sl
invariant trace field} of $G$, denoted by $k(G)$, is defined to be
the field $\Q(\Gamma^{3})$, where $\Gamma^{3}= \langle \gamma^3:
\gamma \in \Gamma \rangle$. It follows from \cite{Mc} that the
invariant trace field is an invariant  of the commensurability
class.

\vspace{2mm}

We say that a subgroup $G$ of ${\rm PU}(2,1)$  is an $\R$-subgroup
if it leaves invariant a totally real geodesic 2-plane in $\ch{2}$.
A subgroup $G$ of ${\rm PU}(2,1)$ is called $\R$-Fuchsian if it is a
discrete $\R$-subgroup. A subgroup $G$ is a $\C$-subgroup if it
leaves invariant a complex geodesic in $\ch{2}$. A subgroup $G$ of
${\rm PU}(2,1)$ is called $\C$-Fuchsian if it is a discrete
$\C$-subgroup.

\vspace{2mm}

By applying Theorem B, we get the following characterization of
discrete non-elementary $\R$-subgroups of ${\rm PU}(2,1)$.

\vspace{2mm}

\noindent \textbf{Theorem C} \textit{Let $G$ be an irreducible
discrete subgroup of ${\rm PU}(2,1)$. Then $G$ is an $\R$-Fuchsian
if and only if the invariant trace field $k(G)$ of $G$ is real.}

\vspace{2mm}

This implies, in particular, that if $G$ is an irreducible discrete
subgroup of ${\rm PU}(2,1)$ whose invariant trace field is real,
then the invariant algebra associated to $G$ is of dimension nine
over $k(G)$. On the other hand, if $G$ is a non-elementary
$\C$-subgroup of ${\rm PU}(2,1)$) ($G$ is reducible in this case)
with real invariant trace field, then the invariant algebra
associated to $G$ is of dimension  four.

\vspace{2mm}

As a corollary of Theorem C, we have the following.

\vspace{2mm}

\noindent \textbf{Theorem D} \textit{Let $G$ be a discrete
non-elementary subgroup of ${\rm PU}(2,1)$ such that the invariant
trace field $k(G)$ of $G$ is real. Then $G$ is either $\R$-Fuchsian
or  $\C$-Fuchsian.}

\vspace{2mm}

We remark that Theorem D can be considered as a complex hyperbolic
analog of a classical result due to B.Maskit, see Theorem G.18 in
\cite{Mas} and Corollary 3.2.5 in \cite{MaR}. Finally, we would like
to mention that some related questions were considered in \cite{Ber,
Fu, Ge, P, V1,V2, W, X}.

\vspace{2mm}

The article  is organized as follows. In Section 1, we review some
basic facts in complex hyperbolic geometry. In Section 2, we prove
our main results.

\section{Complex hyperbolic plane and its isometry group}

Let $V$ be a $3$-dimensional $\mathbb C$-vector space equipped with
a Hermitian form $\herm{-,-}$ of signature $(2,1)$. We denote by
$\P(V)$ the complex projectivization of $V$ and by $\P:V \setminus
\{0\} \rightarrow \P(V)$ a natural projection.

\medskip

Let $V_{-}, V_0, V_{+}$ be the subsets of $V \setminus \{0\}$
consisting of vectors where $\herm{v,v}$ is negative, zero, or
positive respectively. Vectors in $V_0$ are called {\sl null} or
{\sl isotropic}, vectors in $V_{-}$ are called {\sl negative}, and
vectors in  $V_{+}$ are called {\sl positive}. Their projections to
$\P(V)$ are called {\sl isotropic}, {\sl negative}, and {\sl
positive} points respectively.

\medskip

The projective model of the {\sl complex hyperbolic plane} $\ch{2}$
is the set of negative points in  $\P(V)$, that is,
$\ch{2}=\P(V_{-}).$ The boundary $\partial {\ch{2}}=\P(V_{0})$ of
$\ch{2}$  is the $3$-sphere formed by all isotropic  points.

\medskip

The Hermitian form $\herm{-,-}$ defines a metric, the Bergman
metric, on $\ch{2}$,  see \cite{Gol}. Let ${\rm U}(V)$ be the
unitary group corresponding to this Hermitian form. Then the
holomorphic isometry group of $\ch{2}$ is the projective unitary
group ${\rm PU}(V)$, and the full isometry group of $\ch{2}$ is
generated by ${\rm PU}(V)$ and complex conjugation. We denote by
${\rm SU}(V)$ the subgroup of linear transformations in ${\rm U}(V)$
with determinant 1.

\vspace{2mm}

For our purposes it is convenient to work with a basis $e=\{e_1,
e_2, e_3\}$ in $V$ which has the following properties:
$$\herm{e_1,e_1}=0, \ \herm{e_2,e_2}=1, \ \herm{e_3,e_3}=0, \
 \herm{e_1,e_2}=0, \ \herm{e_2,e_3}=0, \ \herm{e_1,e_3}=1.$$

So, in this basis $e_1$ and $e_3$ are isotropic and $e_2$ is
positive. In what follows, we denote by $\C^{2,1}$ the vector space
$V$ equipped with this basis. If the vectors $z=(z_1,z_2,z_3)^T$ and
$w=(w_1,w_2,w_3)^T$ in $\C^{2,1}$ are given by their coordinates in
$e=\{e_1, e_2, e_3\}$, then the Hermitian product $\herm{v,w}$ is
given by
$$\herm{v,w}= z_1 \bar{w}_3 + z_2 \bar{w}_2 + z_3 \bar{w}_1.$$

The use of this basis simplifies essentially matrix computations and
it was successfully applied in a series of works, see, for instance
\cite{CuG1, CuG2, CDGT, DG, E, FGLP, FGP, FP, GuP1, GuP2, W}.
Throughout this paper we will use this basis.

\medskip

Let ${\rm U}(2,1)$ and ${\rm SU}(2,1)$ denote the representations of
${\rm U}(V)$ and ${\rm SU}(V)$ in the basis $e=\{e_1, e_2, e_3\}.$

\medskip

If $A$ is an element of ${\rm SU}(2,1),$ then the matrix  $A$ is
defined by the following simple conditions:

$$\herm{v_1,v_1}=0, \ \herm{v_2,v_2}=1, \ \herm{v_3,v_3}=0, \
 \herm{v_1,v_2}=0, \ \herm{v_2,v_3}=0, \ \herm{v_1,v_3}=1,$$
where $v_1$, $v_2$, $v_3$  denote the vectors defined by the rows of
$A$.

\medskip

Also, we have the following useful formula for the inverse of $A \in
{\rm SU}(2,1):$

$$
A= \left[
\begin{array}{ccc}
a_{11} & a_{12} & a_{13} \\
a_{21} & a_{22} & a_{23} \\
a_{31} & a_{32} & a_{33} \\
\end{array}
\right],\ \
 A^{-1} = \left[
\begin{array}{ccc}
\overline{a}_{33} & \overline{a}_{23} & \overline{a}_{13} \\
\overline{a}_{32} & \overline{a}_{22} & \overline{a}_{12} \\
\overline{a}_{31} & \overline{a}_{21} & \overline{a}_{11} \\
\end{array}
\right].
$$

It is seen that $A^{-1}$ is the Hermitian anti-transpose of $A$.

\medskip

The non-trivial elements of ${\rm PU}(2,1)$ fall into three general
conjugacy types, depending on the number and location of their fixed
points.
\begin{itemize}
\item {\sl Elliptic elements} have a fixed point in $\ch{2}$,
\item {\sl Parabolic elements} have a single fixed point on the boundary of $\ch{2}$,
\item {\sl Loxodromic elements} have exactly two fixed points on the boundary of $\ch{2}$.
\end{itemize}

This exhausts all possibilities, see \cite{Gol} for details.

\medskip

Let $\pi : {\rm SU}(2,1) \rightarrow  {\rm PU}(2,1)$ be a natural
projection.  We call an element $A \in {\rm SU}(2,1)$ loxodromic
(parabolic, elliptic) if its projectivization $\pi(A)$ is loxodromic
(parabolic, elliptic). For instance, any loxodromic element $A \in
{\rm SU}(2,1)$ is conjugate in  ${\rm SU}(2,1)$ to an element of the
following form
$$
A = \left[
\begin{array}{ccc}
\lambda_1 & 0 & 0 \\
0 & \lambda_2 & 0 \\
0& 0 & \lambda_3 \end{array}
 \right],
$$
where $\lambda_1 = \lambda e^{i\varphi}$, $ \lambda_2 =
e^{-2i\varphi},$ $\lambda_3 = \lambda^{-1} e^{i\varphi}$,
$\lambda>0,$ $\lambda \neq 1$, $\varphi \in (-\pi,\pi].$

\medskip

A parabolic element $g \in \pu{2}$ is {\sl unipotent} if it can be
represented by a unipotent element of ${\rm SU}(2,1),$ that is, a
matrix having $1$ as its only eigenvalue. Otherwise, $g$ is {\sl
ellipto-parabolic}. In that case $g$ can be represented by an
element of ${\rm SU}(2,1)$ having a repeated non-real eigenvalue of
norm $1$. Also, $g$ has a unique invariant complex geodesic, see
below.

\medskip

There are two types of totally geodesic submanifolds of $\ch{2}$ of
real dimension two:
\begin{itemize}
\item {\sl Complex geodesics} (copies of $\ch{1}$) have constant
sectional curvature $-1$.
\item {\sl Totally real geodesic 2-planes} (copies of $\rh{2}$)
have constant sectional curvature $-1/4$.
\end{itemize}

\medskip

Any complex geodesic is the intersection of a complex projective
line in $\P(V)$ with $\ch{2},$  and it is uniquely defined by its
polar point, which is positive \cite{Gol}. We recall that a polar
point to a complex projective line $c$ in $\P(V)$ is the
projectivization of the Hermitian orthogonal complement in $V$ of
$\P^{-1}(c)$.

\medskip

Any totally real geodesic 2-plane is the intersection of a totally
real projective plane in $\P(V)$ with $\ch{2}$.

\medskip

We recall that a subspace $S$ of $V_{\R}$, where  $V_{\R}$ is the
real vector space underlying $V$, is {\sl totally real} if and only
if $S$ and its image $\J (S)$ are orthogonal with respect to the
Hermitian product $\herm{-,-}$, see \cite{Gol}. It is easy to show
that $S$ is totally real if and only if the Hermitian product
$\herm{v,u}$ is real for all $v, u \in S.$ An example of a totally
real subspace is the $\R$-linear span of $e=\{e_1, e_2, e_3\}$, the
basis considered above. We call this subspace {\sl a canonical}
totally real subspace.  The projectivization of this space is called
a {\sl canonical totally real} projective 2-plane. Any totally real
2-plane in $\P(V)$ is the image of the canonical totally real
projective 2-plane under an element from the group ${\rm PU}(2,1).$
The stabilizer  of the  canonical totally real subspace in ${\rm
SU}(2,1)$ can be canonically identified with the group ${\rm
SO}(2,1)$.

\medskip

A chain is the boundary of a complex geodesic. An $\R$-circle is the
boundary of a totally real geodesic 2-plane, see \cite{Gol}.

\medskip

Let $G$ be a subgroup of ${\rm PU}(2,1)$. We say that $G$ is a
$\C$-subgroup if it leaves invariant a complex geodesic. $G$ is
called an $\R$-subgroup if it leaves invariant a totally real
geodesic 2-plane. A subgroup $\Gamma$ of ${\rm SU}(2,1)$ is called
$\C$-subgroup if its projectivization is a $\C$-subgroup of ${\rm
PU}(2,1)$. Similarly, a subgroup $\Gamma$ of ${\rm SU}(2,1)$ is
called an $\R$-subgroup if its projectivization is an $\R$-subgroup
of ${\rm PU}(2,1)$. We say that a subgroup $G$ of  ${\rm PU}(2,1)$
is ¨$\C$-Fuchsian if it is a discrete $\C$-subgroup, and $G$ is
called $\R$-Fuchsian if it is a discrete $\R$-subgroup. Typical
examples of $\C$-subgroups of ${\rm SU}(2,1)$ are subgroups of ${\rm
SU}(1,1)$ canonically embedded into ${\rm SU}(2,1)$, and typical
examples of $\R$-subgroups of ${\rm SU}(2,1)$ are subgroups of ${\rm
SO}(2,1)$ canonically embedded into ${\rm SU}(2,1)$.

\medskip

Recall that a subgroup $\Gamma$ of ${\rm SU}(2,1)$ is called {\sl
reducible} if it has an invariant proper $\C$-subspace of $V,$ and
called {\sl irreducible} otherwise. It is clear that a subgroup
$\Gamma$ of ${\rm SU}(2,1)$ is irreducible if and only if $\Gamma$
has no invariant 1-dimensional $\C$-subspaces of $V$.

\medskip

A subgroup $G$ of ${\rm PU}(2,1)$ is called {\sl reducible} if all
elements of $G$ have a common fixed point in their action on the
projective space $\P(V)$. Otherwise, $G$ is called {\sl
irreducible}.

\medskip

A subgroup $G$ of ${\rm PU}(2,1)$ is called {\sl elementary} if it
has a finite orbit in its action on $\ch{2} \cup \partial {\ch{2}}.$
Otherwise, $G$ is {\sl non-elementary}. A subgroup $\Gamma$ of ${\rm
SU}(2,1)$ is called {\sl elementary} if its projectivization is
elementary. Otherwise, $\Gamma$ is {\sl non-elementary}.

\medskip

Clearly, any $\C$-subgroup $G$ of ${\rm PU}(2,1)$ is reducible
because the polar point to the invariant complex geodesic of $G$ is
a common fixed point for all elements of $G$ in their  action on
$\P(V)$. Also, it is clear that a non-elementary $\R$-subgroup of
${\rm PU}(2,1)$ is irreducible.

\section{Trace fields of complex hyperbolic groups}

Let $\Gamma$ be a subgroup of ${\rm SU}(2,1)$. Then the {\sl trace
field} of $\Gamma$, denoted by $\mathbb Q ( {\rm tr} \Gamma)$, is
the field generated by the traces of all the elements of $\Gamma$
over the base field $\mathbb Q$ of rational numbers. For simplicity,
we will denote the trace field of $\Gamma$ by $\Q(\Gamma)$. Of
course, $\Q(\Gamma)$ is a conjugacy invariant. We remark that the
trace field $\Q(\Gamma)$ is also invariant under complex
conjugation, since for $B \in {\rm SU}(2,1)$ we have ${\rm
tr}(B^{-1})= \overline{{\rm tr}(B)}$.

\vspace{3mm}

If $\Gamma$ contains a loxodromic element $A$ with eigenvalues
$\lambda_1 = \lambda e^{i\varphi}$, $\lambda_2 = e^{-i2\varphi}$,
$\lambda_3 = \lambda^{-1}e^{i\varphi},$ then $\Q(\Gamma, \lambda)$
denotes the field generated by $\Q(\Gamma)$ and $\lambda$. Also, for
any field $k$ we denote by ${\rm SU}(2,1,k)$ the intersection of
${\rm SU}(2,1)$ with ${\rm M}(3,k)$.

\begin{lemma} Let $\Gamma=\langle A \rangle$ be the group generated by
a loxodromic element $A \in {\rm SU}(2,1)$ with eigenvalues
$\lambda_1 = \lambda e^{i\varphi}$, $ \lambda_2 = e^{-2i\varphi},$
$\lambda_3 = \lambda^{-1} e^{i\varphi}$. Then $ e^{i\varphi} \in
\Q(\Gamma,\lambda).$
\end{lemma}

\begin{proof} Since $\Q(\Gamma)$ is invariant under complex conjugation, it
follows that ${\rm Re} ({\rm tr}(A))$ and $ \mod{{\rm tr}(A)}^2$ are
in $\Q(\Gamma).$ We have that

$${\rm tr}(A)= \lambda e^{i\varphi} + \lambda^{-1}e^{i\varphi}
+e^{-i2\varphi}$$ and

$${\rm tr}(A^{-1})=\overline{{\rm tr}(A)}= \lambda e^{-i\varphi} +
\lambda^{-1}e^{-i\varphi} +e^{i2\varphi}.$$

\vspace{3mm}

A direct computation shows that

$${\rm Re}( {\rm tr}(A)) = \cos2\varphi + (\lambda +\lambda^{-1})\cos\varphi
= 2\cos^2 \varphi +(\lambda +\lambda^{-1})\cos\varphi - 1$$

and

$$\mod { {\rm tr}(A)}^2 = (\lambda^2 + \lambda^{-2}) + 2(\lambda +
\lambda^{-1})\cos3\varphi + 3.$$

Since $\lambda +\lambda^{-1}\neq 0$, the last formula implies that
$\cos3\varphi \in \Q(\Gamma,\lambda)$.

\vspace{3mm}

Looking at these formulae, one could expect that $\cos\varphi$ is in
a proper extension of $\Q(\Gamma,\lambda)$, but the following trick
shows that, in fact, $\cos\varphi$ is in  $\Q(\Gamma,\lambda)$.

\vspace{3mm}

It is easy to see that the following formula
$$(4\cos^{2} \varphi + 2t\cos\varphi + t^2 - 3)\cos\varphi =
4\cos^{3}\varphi -3\cos\varphi + (2\cos^{2}\varphi +
t\cos\varphi)t$$ is true for all $t \in  \R$ and for all $\varphi
\in (-\pi,\pi].$

\vspace{2mm}

Also, it is elementary to check   that
$$4\cos^{2} \varphi +
2t\cos\varphi + t^2 - 3 \neq 0$$
for all $t>2$ and for all $\varphi
\in (-\pi,\pi].$ This implies that for all  $t>2$ and for all
$\varphi \in (-\pi,\pi]$ we have that

$$\cos\varphi =\frac{4\cos^{3}\varphi -3\cos\varphi +
(2\cos^{2}\varphi + t\cos\varphi)t}{4\cos^{2} \varphi +
2t\cos\varphi + t^2 - 3}.$$
In particular, taking $t= \lambda
+\lambda^{-1}$ and using the identity  $\cos3\varphi =4\cos^3\varphi
- 3\cos\varphi$, we get the formula

$$\cos\varphi =\frac{\cos3\varphi +
(2\cos^{2}\varphi + (\lambda +\lambda^{-1})\cos\varphi)(\lambda
+\lambda^{-1})}{4\cos^{2} \varphi + 2(\lambda
+\lambda^{-1})\cos\varphi + (\lambda +\lambda^{-1})^2 - 3}.$$
One
verifies that this can be re-written as

$$\cos\varphi =\frac{\cos3\varphi + ({\rm Re}( {\rm tr}(A))+1)(\lambda +\lambda^{-1})}
{2{\rm Re}( {\rm tr}(A)) +(\lambda +\lambda^{-1})^2 -1}.$$
Therefore, the above implies that $\cos\varphi \in
\Q(\Gamma,\lambda)$ for all $\varphi \in (-\pi,\pi]$.

\vspace{2mm}

Next, a direct computation shows that

$$i\sin\varphi =\frac{{\rm tr}(A) -\overline{{\rm tr}(A)}}
{2(\lambda + \lambda^{-1} + 2\cos\varphi)}.$$

All the above implies that $ e^{i\varphi} \in \Q(\Gamma,\lambda)$
for all $\varphi \in (-\pi,\pi]$.

\end{proof}

\begin{co} Let $\Gamma=\langle A \rangle$ be the group generated by a loxodromic element $A \in
{\rm SU}(2,1)$ with eigenvalues $\lambda_1 = \lambda e^{i\varphi}$,
 $\lambda_2 = e^{-i2\varphi}$, $\lambda_3 = \lambda^{-1} e^{i\varphi}.$ Then all these
eigenvalues belong to $\Q(\Gamma,\lambda).$
\end{co}

\begin{lemma} Let $\Gamma$ be a subgroup of  ${\rm SU}(2,1)$ containing
a loxodromic element $A={\rm diag}(\lambda_1, \lambda_2,
\lambda_3)$, where $\lambda_1=\lambda e^{i\varphi}$,
$\lambda_2=e^{-i2\varphi}$, $ \lambda_3=\lambda^{-1} e^{i\varphi}.$
Then for any element $B=(b_{ij}) \in \Gamma$, the diagonal elements
 of $B$ are in $\Q(\Gamma,\lambda).$
\end{lemma}

\begin{proof} We write

$$
A = \left[
\begin{array}{ccc}
\lambda_1 & 0 & 0 \\
0 & \lambda_2 & 0 \\
0& 0 & \lambda_3 \end{array}
 \right],\ \ \ \
A^{-1} = \left[
\begin{array}{ccc}
\bar{\lambda}_3 & 0 & 0 \\
0 & \bar{\lambda}_2 & 0 \\
0& 0 & \bar{\lambda}_1 \end{array}
 \right],\ \ \ \
B= \left[
\begin{array}{ccc}
b_{11} & b_{12} & b_{13} \\
b_{21} & b_{22} & b_{23} \\
b_{31} & b_{32} & b_{33} \\
\end{array}
\right].
$$
Then computations show that

$$
\begin{array}{llrlrrrll}
{\rm tr}(B) & = & b_{11} & + & b_{22} & + & b_{33} & = & t_1,\\
{\rm tr}(AB) & = & \lambda_1 b_{11} & + & \lambda_2 b_{22} & + &
\lambda_3
b_{33} &=& t_2,\\
{\rm tr}(A^{-1}B) & = & \bar{\lambda}_3 b_{11} & + & \bar{\lambda}_2
b_{22} & + &\bar{\lambda}_1b_{33} &=& t_3.
\end{array}
$$
Note that every $t_i $ lies in $\Q(\Gamma)$. We consider these
equalities as a system of linear equations in three unknowns
$b_{11}, b_{22}, b_{33}$. Let us show that the matrix $L$ of this
system is nonsingular. We write
$$L= \left[
\begin{array}{ccc}
1 & 1 & 1 \\
\lambda_1 & \lambda_2  & \lambda_3 \\
\bar{\lambda}_3 &  \bar{\lambda}_2 &  \bar{\lambda}_1 \\
\end{array}
\right].
$$
Then a computation gives that
$$\det L = (\lambda_2\bar{\lambda}_1- \lambda_3\bar{\lambda}_2) - (\lambda_1\bar{\lambda}_1-
\lambda_3\bar{\lambda}_3)+ (\lambda_1\bar{\lambda}_2-
\lambda_2\bar{\lambda}_3)= (\lambda^{-2}-\lambda^{2}) + 2(\lambda -
\lambda^{-1})\cos3\varphi .
$$
The equality $\det L =0$ is equivalent to $\cos3\varphi = (\lambda +
\lambda^{-1})/2,$ which is impossible since $\lambda + \lambda^{-1}>
2$.

\vspace{2mm}

It follows from Corollary 2.1 that every coefficient of the system
lies in $\Q(\Gamma,\lambda).$ Solving this system by Cramer's rule,
we conclude that every $b_{ii}$ lies in $\Q(\Gamma,\lambda).$
\end{proof}

\begin{lemma}
Let $\Gamma$ be a subgroup of  ${\rm SU}(2,1)$ containing a
loxodromic element $A={\rm diag}(\lambda_1, \lambda_2, \lambda_3)$,
where $\lambda_1=\lambda e^{i\varphi}$, $\lambda_2=e^{-i2\varphi}$,
$ \lambda_3=\lambda^{-1} e^{i\varphi}.$ Then for any element
$B=(b_{ij}) \in \Gamma$, the products $b_{12}b_{21}, \ b_{13}b_{31},
 \ b_{23}b_{32}$ are in $\Q(\Gamma,\lambda).$
\end{lemma}

\begin{proof} Let $C=BAB$. First, we compute the diagonal elements $c_{ii}$
of the matrix $C$. They are

$$ c_{11}=\lambda_1 b_{11}^2 +\lambda_2 b_{12}b_{21} + \lambda_3
b_{13}b_{31},$$
$$ c_{22}=\lambda_1 b_{12}b_{21} +\lambda_2 b_{22}^2 + \lambda_3
b_{23}b_{32},$$
$$ c_{33}=\lambda_1 b_{13}b_{31} +\lambda_2 b_{23}b_{32} + \lambda_3
b_{33}^2.$$ Then by applying Lemma 2.2, we have that $c_{ii}$ is in
$\Q(\Gamma,\lambda)$.  Also, by the same reason, $b_{ii}$ is in
$\Q(\Gamma,\lambda)$. Therefore, we can re-write these equalities in
the following form

$$\lambda_2 b_{12}b_{21} + \lambda_3 b_{13}b_{31} = t_1,$$
$$\lambda_1 b_{12}b_{21} + \lambda_3 b_{23}b_{32}=t_2,$$
$$\lambda_1 b_{13}b_{31} +\lambda_2 b_{23}b_{32}=t_3,$$
where $t_1, t_2, t_3$ are some elements of $\Q(\Gamma,\lambda).$

We consider these equalities as a system of linear equations in
three unknowns $x_1= b_{12}b_{21}$, $x_2=b_{13}b_{31}$,
$x_3=b_{23}b_{32}$. The matrix $L$ of this system is

$$L= \left[
\begin{array}{ccc}
\lambda_2 & \lambda_3 & 0 \\
\lambda_1 & 0 & \lambda_3 \\
0 & \lambda_1 & \lambda_2 \\
\end{array}
\right].
$$
A short computation shows that $\det L =-2\lambda_1 \lambda_2
\lambda_3 = -2\det A  \neq 0.$ Solving this system by Cramer's rule,
we conclude that every $x_i$ lies in $\Q(\Gamma,\lambda).$
\end{proof}

\begin{lemma}
Let $\Gamma$ be a subgroup of  ${\rm SU}(2,1)$ containing a
loxodromic element $A={\rm diag}(\lambda_1, \lambda_2, \lambda_3)$,
where $\lambda_1=\lambda e^{i\varphi}$, $\lambda_2=e^{-i2\varphi}$,
$ \lambda_3=\lambda^{-1} e^{i\varphi}.$ Then for any element
$B=(b_{ij}) \in \Gamma$, the products $b_{12}\bar{b}_{32}$,
$b_{13}\bar{b}_{31},$ $b_{23}\bar{b}_{21}$ are in
$\Q(\Gamma,\lambda).$
\end{lemma}

\begin{proof}
Let $C=BAB^{-1}$. Let us compute the diagonal elements $c_{ii}$ of
the matrix $C$. They are
$$c_{11}=\lambda_1 b_{11}\bar{b}_{33} +\lambda_2 b_{12}\bar{b}_{32} +
\lambda_3 b_{13}\bar{b}_{31},
$$
$$c_{22}=\lambda_1 b_{21}\bar{b}_{23} +\lambda_2 b_{22}\bar{b}_{22} +
\lambda_3 b_{23}\bar{b}_{21},
$$
$$c_{33}=\lambda_1 b_{31}\bar{b}_{13} +\lambda_2 b_{32}\bar{b}_{12} +
\lambda_3 b_{33}\bar{b}_{11}.
$$
By applying Lemma 2.2, we get that $c_{ii}$ lies in
$\Q(\Gamma,\lambda)$. Since $\lambda$ is real, we have that the
field $\Q(\Gamma,\lambda)$ is invariant under complex conjugation.
Hence, $b_{11}\bar{b}_{33}$, $b_{22}\bar{b}_{22}$,
$b_{33}\bar{b}_{11}$ are all in $\Q(\Gamma,\lambda)$. Therefore, we
can re-write these equalities in the following form

$$\lambda_2 b_{12}\bar{b}_{32} + \lambda_3 b_{13}\bar{b}_{31}=t_1,$$
$$\lambda_1 b_{21}\bar{b}_{23} + \lambda_3 b_{23}\bar{b}_{21}=t_2,$$
$$\lambda_1 b_{31}\bar{b}_{13} +\lambda_2 b_{32}\bar{b}_{12}=t_3,$$
where $t_1, t_2, t_3$ are some elements of $\Q(\Gamma,\lambda).$

\vspace{2mm}

Now let $D=B^{-1}AB$. Again we compute the diagonal elements
$d_{ii}$ of this matrix. They are

$$d_{11}=\lambda_1 b_{11}\bar{b}_{33} +\lambda_2 b_{21}\bar{b}_{23} +
\lambda_3 b_{31}\bar{b}_{13},
$$
$$d_{22}=\lambda_1 b_{12}\bar{b}_{32} +\lambda_2 b_{22}\bar{b}_{22} +
\lambda_3 b_{32}\bar{b}_{12},
$$
$$d_{33}=\lambda_1 b_{13}\bar{b}_{31} +\lambda_2 b_{23}\bar{b}_{21} +
\lambda_3 b_{33}\bar{b}_{11}.
$$
By applying the above arguments, we re-write these equalities in the
following form

$$\lambda_2 b_{21}\bar{b}_{23} + \lambda_3 b_{31}\bar{b}_{13}=s_1,$$
$$\lambda_1 b_{12}\bar{b}_{32} + \lambda_3 b_{31}\bar{b}_{12}=s_2,$$
$$\lambda_1 b_{13}\bar{b}_{32} +\lambda_2 b_{23}\bar{b}_{21}=s_3,$$
where $s_1, s_2, s_3$ are some elements of $\Q(\Gamma,\lambda).$

\vspace{2mm}

We consider the first and the conjugate third equality defined by
$C$ and the conjugate first equality defined by $D$. So, we have the
following equalities

$$\lambda_2 b_{12}\bar{b}_{32} + \lambda_3 b_{13}\bar{b}_{31}=t_1,$$
$$\bar{\lambda}_1 b_{13}\bar{b}_{31} +\bar{\lambda}_2 b_{12}\bar{b}_{32}=\bar{t}_3,$$
$$\bar{\lambda}_2 b_{23}\bar{b}_{21} + \bar{\lambda}_3 b_{13}\bar{b}_{31}=\bar{s}_1.$$
These equalities define a system of linear equations in three
unknowns $x_1= b_{12} \bar{b}_{32}$, $x_2=b_{13}\bar{b}_{31}$,
$x_3=b_{23} \bar{b}_{21}$. The matrix $L$ of this system is
$$L= \left[
\begin{array}{ccc}
\lambda_2 & \lambda_3 &0 \\
\bar{\lambda}_2 & \bar{\lambda}_1 & 0 \\
0 & \bar{\lambda}_3 & \bar{\lambda}_2 \\
\end{array}
\right].$$
The determinant $\det L = \lambda_2(\bar{\lambda}_1
\bar{\lambda}_2) -\lambda_3 \bar{\lambda}_{2}^2 =
e^{2i\varphi}(e^{-3i\varphi}\lambda - e^{3i\varphi}\lambda^{-1}).$
It is seen that $\det L=0$ if and only if $\lambda^2=e^{6i\varphi}$.
This is impossible since $\lambda>0, \lambda \neq 1.$  Solving this
system by Cramer's rule, we conclude that every $x_i$ lies in
$\Q(\Gamma,\lambda).$
\end{proof}

\vspace{2mm}

The following proposition is crucial  in the proof of our main
result.

\begin{prop} Let $\Gamma = \langle A, B\rangle$  be an irreducible subgroup of
${\rm SU}(2,1)$, where $A$ is a  loxodromic element with eigenvalues
$\lambda_1 = \lambda e^{i\varphi},$ $\lambda_2 =e^{-2i\varphi}$,
$\lambda_3 = \lambda^{-1}e^{i\varphi}$. Then $\Gamma$ is conjugate
in ${\rm SU}(2,1)$ to a subgroup of ${\rm SU}(2,1,
\Q(\Gamma,\lambda))$.
\end{prop}

\begin{proof} We will show that one needs at most two conjugations to get the
result we need.

\vspace{2mm}

First, by applying a suitable conjugation in ${\rm SU}(2,1)$, we may
assume that

$$
A = \left[
\begin{array}{ccc}
\lambda_1 & 0 & 0 \\
0 & \lambda_2 & 0 \\
0& 0 & \lambda_3 \end{array}
 \right],\ \ \ \
B= \left[
\begin{array}{ccc}
b_{11} & b_{12} & b_{13} \\
b_{21} & b_{22} & b_{23} \\
b_{31} & b_{32} & b_{33} \\
\end{array}
\right].
$$

\vspace{2mm}

We denote by $v_1=(b_{11}, b_{12}, b_{13})^T,$  $v_2= (b_{21},
b_{22}, b_{23})^T,$ $v_3=(b_{31}, b_{32},b_{33})^T$ the vectors
defined by the rows of the matrix $B$.

\vspace{2mm}

Let $X \in {\rm SU}(2,1)$ be a loxodromic element (elliptic if r=1)
such that
$$
X= \left[
\begin{array}{ccc}
re^{i\alpha}& 0 & 0 \\
0 & e^{-2i\alpha}& 0 \\
0& 0 & r^{-1}e^{i\alpha} \end{array}
 \right].\ \ \ \
$$
Then $XAX^{-1}=A$ and $XBX^{-1}$ is

$$XBX^{-1}= \left[
\begin{array}{ccc}
b_{11} & re^{3i\alpha}b_{12} & r^{2}b_{13} \\
r^{-1}e^{-3i\alpha}b_{21} & b_{22} & re^{-3i\alpha}b_{23} \\
r^{-2}b_{31} & r^{-1}e^{3i\alpha}b_{32} & b_{33} \\
\end{array}
\right].
$$

\vspace{2mm}

If the entries $b_{12},$ $ b_{32}$ of $B$ are all equal to $0$, then
the group $\Gamma$ is not irreducible since in this case $A$ and $B$
have a common invariant complex line spanned by the vector
$(0,1,0)^{T}.$ In this case, $\Gamma$ is a $\C$-group. Therefore, at
least one of the numbers $b_{12}, b_{32}$ is not equal to $0$. Let
us suppose that $b_{12}\neq 0$. In this case, by normalizing the
elements $A$ and $B$ using  $X$, we may assume without loss of
generality that $b_{12}=1$.

\vspace{2mm}

From the above results we know that $b_{ii},$ $b_{12}b_{21}$,
$b_{13}b_{31}$, $b_{23}b_{32}$, $b_{12}\bar{b}_{32}$,
$b_{13}\bar{b}_{31},$ $b_{23}\bar{b}_{21}$ are all in
$\Q(\Gamma,\lambda)$.

\vspace{2mm}

Since $b_{12}=1$, we get that $b_{21}$ and $b_{32}$ are in
$\Q(\Gamma, \lambda)$. Let us first consider the case $b_{21} = 0$.
Then we have that $b_{31}\neq 0$. Indeed, if  $b_{31}=0$, then $A$
and $B$ have a common invariant complex line spanned by the vector
$(1,0,0)^{T}.$ This is impossible because $\Gamma$ is irreducible.
Let us assume now that $b_{32}=0$. Since the vectors $v_2$ and $v_3$
are orthogonal, we get that $ \herm{v_2,v_3} = b_{23}
\bar{b}_{31}=0$. So, in this case, $b_{23}=0$. This implies that
$b_{22}\neq 0$. Hence, $\herm{v_1,v_2} =\bar{b}_{22} \neq 0,$ a
contradiction. Therefore, we have that $b_{32}\neq 0$. It follows
that $b_{23}$ lies in $\Q(\Gamma, \lambda)$. Since the vectors $v_1$
and $v_2$ are orthogonal, we have that $b_{23} \neq 0.$ Then, by
considering the equality $\herm{v_2,v_3}= b_{22}\bar{b}_{32}+
b_{23}\bar{b}_{31}=0$, we conclude that $b_{31}$ lies in $\Q(\Gamma,
\lambda)$. Since  $b_{31}\neq 0,$ we get that $b_{13}$ lies in
$\Q(\Gamma, \lambda)$. This shows that all the entries of the matrix
$B$ are in $\Q(\Gamma, \lambda)$. Therefore, we proved the
proposition in this case.

Now let us assume that $b_{21}\neq 0.$ It follows immediately that
$b_{23}$ lies in $\Q(\Gamma, \lambda)$.

We write $\herm{v_1,v_2}=b_{11}\bar{b}_{23} +\bar{b}_{22} +
b_{13}\bar{b}_{21}=0.$ Since $b_{21}\neq 0$, this implies that
$b_{13}$ is in  $\Q(\Gamma, \lambda)$. If  $b_{13} \neq 0$, we get
that $b_{31}$ is in  $\Q(\Gamma, \lambda).$ Suppose that $b_{13} =
0$. In this case, $b_{23}\neq 0$, since otherwise $A$ and $B$ have a
common invariant complex line spanned by the vector $(0,0,1)^{T}.$
This is impossible because $\Gamma$ is irreducible.

We write $\herm{v_2,v_3}= b_{21}\bar{b}_{33}+ b_{22}\bar{b}_{32} +
b_{23}\bar{b}_{31}=0.$ Since $b_{23}\neq 0$, we get that $b_{31}$
lies in $\Q(\Gamma, \lambda)$.

Summarizing everything, we get that in the case $b_{12} \neq 0$
after conjugation all the entries of the matrix $B$ lie in
$\Q(\Gamma, \lambda).$ The case $b_{12} = 0$, $b_{32}\neq 0$ is
similar to the previous one.
\end{proof}

\vspace{2mm}

Next we prove that any irreducible subgroup of ${\rm SU}(2,1)$
contains a loxodromic element. First, we will prove the following
lemma.

\begin{lemma} Let $\Gamma$ be a non-elementary subgroup of
 ${\rm SU}(2,1).$  If $\Gamma$ contains a parabolic element, then  $\Gamma$
 contains a loxodromic element.
\end{lemma}

\begin{proof}
Let $B \in \Gamma$ be parabolic. Then $B$ has a unique invariant
isotropic complex line in $V$. By normalizing  $\Gamma$ in ${\rm
SU}(2,1),$ we may assume without loss of generality that this line
is spanned by the vector $(1,0,0)^{T}.$ This implies that $B$ has
the following form:

$$B= \left[
\begin{array}{ccc}
b_{11}& b_{12} & b_{13} \\
0 & b_{22} & b_{23} \\
0 & b_{32} & b_{33} \\
\end{array}
\right].
$$

As before, we denote by $v_1=(b_{11}, b_{12}, b_{13})^T,$  $v_2=
(b_{21}, b_{22}, b_{23})^T,$ $v_3=(b_{31}, b_{32},b_{33})^T$ the
vectors defined by the rows of the matrix $B$.

\vspace{2mm}

We consider two cases: (1) $b_{12}=0,$ (2)  $b_{12} \neq 0.$

\vspace{2mm}

First, we consider the case $b_{12}=0$. We have that $
\herm{v_2,v_2} = |b_{22}|^2=1$. From this and the equality $
\herm{v_2,v_3}=0$, we get that $b_{32}=0.$ Then the equality $
\herm{v_1,v_2}=b_{11} \bar{b}_{23}= 0$ implies that $b_{23}=0.$ From
the equality $ \herm{v_1,v_3} =1$, we get that $b_{11}
\bar{b}_{33}=1.$ Since $B$ is parabolic, we have that $b_{13} \neq
0$ and that $|b_{11}|=1$. Hence $|b_{33}|=1$. We write
$b_{11}=e^{i\varphi}$ and $b_{13}=r e^{i\theta}$, $r>0$. Then the
equality $\herm{v_1,v_1}=0$ implies that $b_{13}=ri e^{i\varphi}$.
Therefore, when $b_{12}=0,$ the element $B$ after a suitable
normalization of  $\Gamma$ has the following form:

$$B= \left[
\begin{array}{ccc}
e^{i\varphi}& 0 & ri e^{i\varphi} \\
0 & e^{-2i\varphi} & 0 \\
0 & 0 & e^{i\varphi} \\
\end{array}
\right].
$$

We remark that if $\varphi=0$, then $B$ is unipotent and $B$ is
ellipto-parabolic otherwise.

\vspace{2mm}

Next, we consider the case $b_{12} \neq 0$. Using the same arguments
as in the first case, we get that $|b_{22}|=1$ and $b_{32}=0$. Since
$b_{12} \neq 0$,  by normalizing  $\Gamma$   using the element $X$
defined in the proof of Proposition 2.1, we may assume without loss
of generality that $b_{12}=1$. The equality $ \herm{v_1,v_3} =1$
implies that $b_{11}  \bar{b}_{33}=1.$ Note that in this case $B$ is
always unipotent. Hence $b_{11}=1$. This implies that
$b_{22}=b_{33}=1$. Then it follows from  the equality $
\herm{v_1,v_2} =0$ that $b_{23}=-1$. Now let us consider the
equality $\herm{v_1,v_1} = \bar{b}_{13} + b_{13} + 1 = 0$. It easy
to see that this equality is true if and only if $b_{13}= -1/2 +
si$, where $s$ is real.  Summarizing everything, we get that in the
case $b_{12} \neq 0$ after a suitable normalization of  $\Gamma$ the
element $B$ has the following form:

$$B= \left[
\begin{array}{ccc}
1& 1& \tau\\
0 & 1 & -1 \\
0 & 0 & 1 \\
\end{array}
\right],$$
where $\tau = -1/2 + si$, $s \in \R.$

\vspace{2mm}

Easy induction shows that for any $n \in \mathbb N$ in the first
case

$$B^n= \left[
\begin{array}{ccc}
e^{ni\varphi}& 0 & e^{ni\varphi}i n r \\
0 &e^{-2ni\varphi}  & 0 \\
0 & 0 & e^{ni\varphi} \\
\end{array}
\right]
$$
and

$$B^n= \left[\vspace{2mm}
\begin{array}{ccc}
1& n & n \tau +f(n) \\
0 & 1 & -n \\
0 & 0 & 1 \\
\end{array}
\right].
$$
in the second case, where the function $f:\mathbb N \rightarrow
\mathbb Z$ is defined by the following conditions: $f(1)=0$ and
$f(n+1)=f(n)-n$. It is easy to show that $f(n) = (1-n)n/2$.

\vspace{3mm}

Since $\Gamma$ is non-elementary it contains an element $C$ which
does not leave invariant the complex line spanned  by the vector
$(1,0,0)^{T}.$ Let $C=(c_{ij})$. Then we have that $c_{31}\neq 0$.
Easy computation shows that

$${\rm tr}(B^n C)= c_{11} e^{ni\varphi} +c_{31}e^{ni\varphi}i n r +
c_{22} e^{-2ni\varphi} + c_{33} e^{ni\varphi}$$ in the first case,
and

$${\rm tr}(B^n C) = c_{11} + c_{21} n + c_{31}(n \tau +f(n)) +
c_{22} - c_{32}n + c_{33}$$ in the second one.

\vspace{3mm}

Hence, in the first case

$$|{\rm tr}(B^n C)|= |c_{31}e^{ni\varphi}i n r - (-c_{11} e^{ni\varphi}
-c_{22} e^{-2ni\varphi} - c_{33} e^{ni\varphi})|$$

$$\geq |c_{31}e^{ni\varphi}i n r| - |(-c_{11} e^{ni\varphi} -c_{22}
e^{-2ni\varphi} - c_{33} e^{ni\varphi})|
$$

$$= |c_{31}| n r  - |(-c_{11} e^{ni\varphi} -c_{22}
e^{-2ni\varphi} - c_{33} e^{ni\varphi})|.
$$.

Since $c_{31} \neq 0,$ and $r>0$, this inequality implies that
$|{\rm tr}(B^n C)| $ tends to infinity when $n$ tends to infinity.

\vspace{2mm}

In the second case, we have that

$$|{\rm tr}(B^n C)| =|c_{11} + c_{21} n + c_{31}(n \tau +f(n)) +
c_{22} - c_{32}n + c_{33}|$$

$$\geq |c_{31}f(n)| - |-c_{31}n \tau -c_{11} - c_{21} n  - c_{22} +
c_{32}n  -c_{33}|$$

$$=|c_{31}||f(n)| - |-c_{31}n \tau -c_{11} - c_{21} n  - c_{22} +
c_{32}n  -c_{33}|.$$
Since $c_{31} \neq 0$ and $f(n)$ is quadratic,
this inequality implies that $|{\rm tr}(B^n C)| $ tends to infinity
when $n$ tends to infinity.

\vspace{2mm}

Thus, in both cases, we have that there exists $n_0$ such that
$|{\rm tr}(B^n C)|>3$ for all $n>n_0$. Then using Goldman's
classification of the elements in ${\rm SU}(2,1)$ \cite{Gol}, we get
that the elements $B^n C$ are loxodromic for all $n>n_0$. This
proves the lemma.
\end{proof}

\vspace{2mm}

We are very grateful to John Parker for his help in the proof of
this lemma.

\vspace{3mm}

In what follows, we will need the following fundamental result due
to Chen and Greenberg \cite{CG}:

\begin{prop} Let $\Gamma$ be a subgroup of ${\rm
SU}(2,1).$ Consider the natural action of  $\Gamma$ on $\ch{2} \cup
\partial {\ch{2}}$. If there is no
point in $\ch{2} \cup \partial {\ch{2}}$ or proper totally geodesic
submanifold in  $\ch{2}$ which is invariant under $\Gamma$, then
$\Gamma$ is either discrete or dense in ${\rm SU}(2,1)$.
\end{prop}

\vspace{2mm}

By applying this result and Lemma 2.5, we prove the following
proposition.

\begin{prop} Let $\Gamma$ be an irreducible subgroup of ${\rm
SU}(2,1)$. Then $\Gamma$ contains a loxodromic element.
\end{prop}

\begin{proof} Let us consider the natural action of $\Gamma$ on the
projective space $\P(V)$ (or equally the action of its
projectivization). We know that the only proper totally geodesic
submanifolds of $\ch{2}$ are either geodesics, or complex geodesics,
or totally real geodesic 2-planes \cite{CG, Gol}. Note that if there
is a geodesic in $\ch{2}$ which is invariant under $\Gamma$, then
there is a complex geodesic $c$ in $\ch{2}$ which is also invariant
under $\Gamma$ (this complex geodesic $c$ is spanned by the geodesic
in question). Therefore, the polar point to $c$ is invariant under
$\Gamma$ for its action on the projective space $\P(V)$. Hence, if
$\Gamma$ is irreducible, then either $\Gamma$ has no invariant
proper totally geodesic submanifolds or $\Gamma$ is an $\R$-subgroup
of ${\rm SU}(2,1)$.

First, we consider the case when $\Gamma$ has no invariant totally
real geodesic 2-planes. By applying Proposition 2.2, we get that
$\Gamma$ is either discrete or dense in ${\rm SU}(2,1)$. If $\Gamma$
is dense in ${\rm SU}(2,1)$, it contains a loxodromic element since
the set of loxodromic elements is open in ${\rm SU}(2,1)$
\cite{Gol}, and we are done. Now let us assume that $\Gamma$ is a
discrete subgroup of ${\rm SU}(2,1)$. Suppose that $\Gamma$ has no
loxodromic elements. If $\Gamma$ contains only elliptic elements,
then it is easy to see that $\Gamma$ is finite, and then, using the
arguments similar to those in \cite{Rat}, we conclude that there is
a point in $\ch{2}$ which is invariant under $\Gamma$, a
contradiction. Therefore, $\Gamma$ contains a parabolic element
$\gamma$. Since $\Gamma$ is irreducible, it contains an element not
fixing the unique fixed point of $\gamma$. By applying Lemma 2.5, we
get that $\Gamma$ contains a loxodromic element, a contradiction.

Now, let $\Gamma$ be an $\R$-subgroup. Then $\Gamma$ is conjugate in
${\rm SU}(2,1)$ to a subgroup of  ${\rm SO}(2,1)$. Since $\Gamma$
irreducible, the result follows from the plane real hyperbolic
geometry \cite{B, Rat}.
\end{proof}

\vspace{3mm}

\begin{co} Let  $G$ be an irreducible subgroup of  ${\rm
PU}(2,1)$. Then $G$ is non-elementary.
\end{co}

\begin{proof} Assume that  $G$ is elementary. Let $p \in \ch{2} \cup \partial {\ch{2}}$
be a point such that the orbit $G(p)$ of $p$ is finite. If $p \in
\ch{2}$, then it follows from the proof of Proposition 2.3 that
there is a point in $\ch{2}$ which is invariant under $\Gamma$, a
contradiction. Next, let us suppose that $p \in \partial {\ch{2}}$.
By applying Proposition 2.3, we have that $G$ contains a loxodromic
element. Since $G(p)$ is finite, this implies that $G(p)$ consists
of two distinct points. Let $G(p)=\{p_1, p_2\}$ and $c$ be the
unique complex geodesic spanned by $\{p_1, p_2\}$. Then $c$ is
invariant under $G$, and, therefore, $G$ leaves invariant its polar
point, a contradiction.
\end{proof}

\vspace{3mm}

Thus, we proved that any irreducible subgroup of ${\rm SU}(2,1)$
contains a loxodromic element. With this, we now prove our main
result: if $\Gamma$ is an irreducible subgroup of ${\rm SU}(2,1)$,
then $\Gamma$  is conjugate in ${\rm SU}(2,1)$ to a subgroup of
${\rm SU}(2,1,\Q(\Gamma, \lambda)),$ where $\lambda$ is the absolute
value of an eigenvalue of any loxodromic element of $\Gamma$ with
$|\lambda| \neq 1$. In particular, this implies that $\Gamma$ can be
defined over the field $\Q(\Gamma, \lambda)$.

\begin{theorem} Let $\Gamma$ be an irreducible subgroup of
 ${\rm SU}(2,1)$. Let $A \in \Gamma$ be any loxodromic element with eigenvalues
$\lambda_1 = \lambda e^{i\varphi},$  $\lambda_2 = e^{-2i\varphi}$,
$\lambda_3 = \lambda^{-1}e^{i\varphi}$. Then $\Gamma$ is conjugate
in ${\rm SU}(2,1)$ to a subgroup of ${\rm SU}(2,1,\Q(\Gamma,
\lambda)).$
\end{theorem}

\begin{proof} First, we show that $\Gamma$
contains two loxodromic elements $A_1$ and $A_2$ having the same
eigenvalues as $A$ (in fact, we show that $A_1$ and $A_2$ are
conjugate to $A$ in $\Gamma$) such that the subgroup $\Gamma_0 =
\langle A_1, A_2 \rangle$ generated by $A_1$ and $A_2$ is
irreducible. In what follows, we consider the natural action of
$\Gamma$ on the projective space $\P(V)$. We remark that $A$ has
three fixed points $p_1 , p_2 , p_3$ for its action on the
projective space: $p_1 , p_2$ are isotropic and $p_3$ is positive,
$p_3$ is the polar point to the complex projective line $\alpha$
spanned by $p_1$ and $p_2$. If all the elements of $\Gamma$ fix the
point $p_3$, then $\Gamma$ is reducible. So, there exists an element
$B$ of $\Gamma$ which does not fix $p_3$. Let $C=BAB^{-1}$ and
$q_3=B(p_3)$. Then $C(q_3)=q_3$. Note that $C$ is loxodromic, hence
$C$ does not fix $p_3$. Let $q_1$ and $q_2$ be the isotopic fixed
point of $C$. If the sets $\{p_1, p_2\}$ and $\{q_1, q_2\}$ are
disjoint, taking $A_1=A$ and $A_2=C$ we are done, since in this case
the elements $A$ and $C$ have no common fixed points for their
action on $\P(V)$. So, let us consider the case when these sets have
non-empty intersection, that is, when the elements $A$ and $C$ have
a common isotropic fixed point. If $\{p_1, p_2\}= \{q_1, q_2\}$,
then the complex projective line $\alpha$ is invariant under $C$,
and, therefore, $C(p_3)=p_3$, a contradiction. This implies that $A$
and $C$ may have only one common isotropic fixed point. One may
assume without loss of generality that $p_1=q_1$ is a unique common
isotropic fixed point of $A$ and $C$. Then since $\Gamma$ is
irreducible, there exists an element $D \in \Gamma$ which does not
fix the point $p_1$. Let $E=DAD^{-1}$. We have that $E$ is
loxodromic and $E(p_1) \neq p_1$. If the invariant complex geodesic
$\beta$ of $E$ is not equal to $\alpha$, then taking $A_1=A$ and
$A_2=E$, we are done. If $\alpha =\beta$, we take $A_1=C$ and
$A_2=E$.

\vspace{2mm}

Now, by applying Proposition 2.1, we get that there exists $f \in
{\rm SU}(2,1)$ such that $\Gamma_{0}^{*} = f \Gamma_0 f^{-1}$ is a
subgroup of ${\rm SU}(2,1, \Q(\Gamma_{0}, \lambda)).$ Let
$\Gamma^{*} = f \Gamma f^{-1}.$ Then $\Gamma_{0} ^{*}$ is a subgroup
of $\Gamma^{*}$. In order to show that $\Gamma^{*}$ is a subgroup of
${\rm SU}(2,1,\Q(\Gamma,\lambda)),$ we apply Burnside's density
theorem, stated as Theorem 16'' in I.Kaplansky \cite{Kap}, and the
trick which one can find in the proof of Theorem B in I.Kaplansky
\cite{Kap}. According to I.Kaplansky this trick is due to C.Procesi.

\vspace{2mm}

Since $\Gamma_{0}^{*}$ is irreducible, by applying Burnside's
density theorem, we get that $\Gamma_{0}^{*}$ contains a basis of
${\rm M}(3,\C)$ over $\C$. Let $S=\{S_1, S_2, \ldots , S_9\}$ be
such a basis. Then it follows that for any element $\gamma \in
\Gamma^{*}$ there exist complex numbers $c_1, c_2, \ldots, c_9$ such
that
$$\gamma = c_1 S_1 + c_2  S_2 + \ldots +c_9 S_9.$$

Let ${\rm Tr}(A,B)=(A,B)$ denote the trace form on ${\rm M}(3,\C)$
so that

$$(A,B)={\rm Tr}(A,B) = {\rm  tr}(AB).$$

It is well-known that ${\rm Tr}$ is a non-degenerate symmetric
bilinear form.

\vspace{2mm}

It follows from the equality

$$\gamma = c_1 S_1 + c_2  S_2 + \ldots +c_9 S_9$$
that

$$(\gamma,S_i) = c_1 (S_1,S_i) + c_2  (S_2,S_i) + \ldots +c_9 (S_9,S_i)$$
for all $i=1, 2, \ldots, 9.$

We consider these equalities as a system of linear equations in
unknowns $c_1, c_2, \ldots, c_9$. We have that for all $i,j=1, 2,
\ldots, 9,$ the product $(S_i,S_j)$ is in $\Q(\Gamma_{0}, \lambda)$,
and the product $(\gamma,S_i)$ is in $\Q(\Gamma, \lambda)$.
Therefore, this implies that every coefficient of the system lies in
$\Q(\Gamma,\lambda).$ Since the form ${\rm Tr}$ is non-degenerate,
this system is non-singular. Solving this system by Cramer's rule,
we conclude that every $c_i$ lies in $\Q(\Gamma,\lambda).$ This
implies that every $\gamma \in \Gamma^{*}$ lies in ${\rm
SU}(2,1,\Q(\Gamma,\lambda)).$  From this, we conclude that
$\Gamma^{*}$ is a subgroup of ${\rm SU}(2,1,\Q(\Gamma,\lambda)).$
\end{proof}

\vspace{2mm}

As a corollary of this theorem, we get the  following.

\begin{theorem}
Let $\Gamma$ be an irreducible subgroup of ${\rm SU}(2,1)$ such that
$\Q(\Gamma)$ is a subset of $\R$, then $\Gamma$ is conjugate in
${\rm SU}(2,1)$ to a subgroup of ${\rm SO}(2,1)$.
\end{theorem}

\begin{proof} Since  $\Gamma$ is irreducible, it follows from Proposition
2.3 that $\Gamma$  contains a loxodromic element $A$. Let $\lambda_1
= \lambda e^{i\varphi},$ $\lambda_2 = e^{-2i\varphi}$, $\lambda_3 =
\lambda^{-1}e^{i\varphi}$ be its eigenvalues. It is easy to see that
in this case all the eigenvalues of $A$ are real: $A$ is either
hyperbolic or loxodromic whose elliptic part is of order 2.  Since
the field $\Q(\Gamma)$ is real, the field $\Q(\Gamma, \lambda)$ is
also real. So, the result follows from Theorem 2.1.
\end{proof}

\vspace{3mm}

We would like to stress that in Theorem 2.1 and Theorem 2.2 we do
not assume that the group $\Gamma$ is discrete.

\vspace{2mm}

Next we show that the conclusions of Theorem 2.1 and Theorem 2.2 are
not true if $\Gamma$ is reducible. Indeed, let us consider the
embedding of ${\rm SL}(2,\R)$ in ${\rm SU}(2,1)$ given by

$$
\left[
\begin{array}{cc}
a & b \\
c & d
\end{array}
\right] \longrightarrow \left[
\begin{array}{ccc}
a & 0 & - i b\\
0 & 1 & 0\\
i c & 0 & d
\end{array}
\right].
$$
This defines a faithful representation  of $\rm{SL}(2,\R)$ in  ${\rm
SU}(2,1)$. Let $\Gamma$ be the image of $\rm{SL}(2,\R)$ under this
embedding. Then $\Gamma$ is reducible  because the complex line
spanned by the vector $(0,1,0)^T$ is invariant under $\Gamma$. In
fact, $\Gamma$ is a $\C$-subgroup of ${\rm SU}(2,1)$. Also, we have
that the trace field of $\Gamma$ is real. It is easy to see that
$\Gamma$ cannot be conjugate in ${\rm SU}(2,1)$ to a subgroup of
${\rm SO}(2,1)$. Moreover, if we consider the image of
$\rm{SL}(2,\Z)$ under this embedding, we get an example of a
non-elementary discrete subgroup of ${\rm SU}(2,1)$ whose trace
field is $\Q$ but which is not conjugate in ${\rm SU}(2,1)$ to a
subgroup of ${\rm SO}(2,1)$.

\vspace{2mm}

Now we consider the case of subgroups of ${\rm PU}(2,1)$. We would
like to find conditions under which a subgroup $G$ of ${\rm
PU}(2,1)$  leaves invariant a totally real geodesic 2-plane in
$\ch{2}$. It is natural to ask if it is possible or not to get an
answer in terms of the traces of lifts of elements of $G$ to ${\rm
SU}(2,1)$.

\vspace{2mm}

Let $\pi : {\rm SU}(2,1) \rightarrow  {\rm PU}(2,1)$ be a natural
projection. Let $G$ be a subgroup of ${\rm PU}(2,1)$. Then the {\sl
trace field} of $G$, denoted $\Q(G)$, is the field $\Q(\Gamma)$,
where $\Gamma = \pi^{-1}(G).$ We remark that this field is never
real.

\vspace{2mm}

The following example is useful to understand the problem. Let
$\Gamma$ be a subgroup of ${\rm SO}(2,1)$ such that the restriction
$\pi : \Gamma \rightarrow {\rm PU}(2,1)$ is a monomorphism. Let $G=
\pi(\Gamma)$. We have that $\Gamma$ is a lift of $G$. One could
define an "invariant" trace field of $G$ as the trace field of
$\Gamma$. And, in this case, this field is real! So, it would be
natural to define an "invariant" trace field of a subgroup $G$ of
${\rm PU}(2,1)$ which has a lift to ${\rm SU}(2,1)$ as the trace
field of its lift. Unfortunately, the problem when a subgroup $G$ of
${\rm PU}(2,1)$ has a lift to ${\rm SU}(2,1)$ is still open, and it
seems very difficult. To our knowledge, the only results in this
direction are in  \cite{GKL}, in the case when $G$ is isomorphic to
the fundamental group of a closed orientable surface. Some examples
when $G$ has no lift, the reader can find in \cite{AGG}. The lifting
problem in the case of real hyperbolic geometry of dimension three
was completely solved in \cite{Cul}.

\vspace{2mm}

In \cite{R}, \cite{MaR}, \cite{V2} the invariant trace field was
defined for subgroups of $\rm{PSL}(2,\C),$ see also \cite{Mc}, for
the case of ${\rm SU}(n,1).$ We follow their ideas to define an
invariant trace field for subgroups of ${\rm PU}(2,1)$.

\vspace{2mm}

Let $G$ be a subgroup of ${\rm PU}(2,1)$ and $\Gamma = \pi^{-1}(G)$.
Then the  {\sl invariant trace field} of $G$, denoted by $k(G)$, is
defined to be the field $\Q(\Gamma^{3})$, where $\Gamma^{3}= \langle
\gamma^3: \gamma \in \Gamma \rangle$. It follows from \cite{Mc} that
the invariant trace field is an invariant  of the commensurability
class.

\vspace{2mm}

If $A \in {\rm SU}(2,1)$, then we have the following trace identity

$$ \rm{tr}(A^3)= (\rm{tr}(A))^3 -3\rm{tr}(A)\rm{tr}(A^{-1}) +3,$$
see, for instance, \cite{P}. Since an element of ${\rm PU}(2,1)$ has
three lifts to ${\rm SU}(2,1)$ which differ by a cube root of unity,
it follows from this formula that if $G$ is an $\R$-subgroup of
${\rm PU}(2,1)$, then the invariant trace field of $G$ is real.

\vspace{2mm}

By applying Theorem 2.2, we get the following characterization of
discrete  $\R$-subgroups of ${\rm PU}(2,1)$.

\begin{theorem}
Let $G$ be an irreducible discrete subgroup of  ${\rm PU}(2,1)$.
Then $G$ is an $\R$-subgroup if and only if the invariant trace
field $k(G)$ of $G$ is real.
\end{theorem}

\begin{proof} Let $\Gamma = \pi^{-1}(G)$. Then it is clear that $\Gamma$ and
 $\Gamma^{3}$ are irreducible. Let $H=\pi(\Gamma^{3})$.
It follows from Theorem 2.2 that $\Gamma^{3}$ is conjugate in ${\rm
SU}(2,1)$ to a subgroup of ${\rm SO}(2,1)$. This implies that $H$
leaves invariant a totally real geodesic 2-plane $L$ in $\ch{2}$.
Since $\Gamma^{3}$ is a normal subgroup of $\Gamma$, it follows that
$H$ is a normal subgroup of $G$. Since the group $G$ is irreducible,
it follows from Corollary 2.2 that $G$ is non-elementary. Therefore,
the normality implies that the limit set of $H$ is equal to the
limit set of $G$, see, for instance, \cite{CG}. We remark that the
limit set of $H$ is contained in the boundary of $L$. Moreover,
using the fact that $G$ is non-elementary, we have that the set of
fixed points of loxodromic elements of $G$ is dense in its limit
set. This implies that $L$ is invariant with respect to $G$.
Therefore, $G$ is an $\R$-subgroup of ${\rm PU}(2,1)$.
\end{proof}

\vspace{2mm}

As an immediate corollary we get the following theorem.

\begin{theorem} Let $G$ be a discrete non-elementary subgroup of ${\rm
PU}(2,1)$ such that the invariant trace field $k(G)$ of $G$ is real.
Then $G$ is either $\R$-Fuchsian or  $\C$-Fuchsian.
\end{theorem}

\begin{proof} If $G$ is irreducible, then it follows from Theorem
2.3 that $G$ is $\R$-Fuchsian. Now let us assume that $G$ is
reducible. Let $p$ be a common fixed point of all elements of $G$ in
their action on the projective space $\P(V)$. Since $G$ is
non-elementary, we have that $p$ cannot be in  $\ch{2} \cup
\partial {\ch{2}}$. So, $p$ is positive. This implies that $G$
leaves invariant the complex geodesic whose polar point is $p$.
Therefore, the group $G$ is $\C$-Fuchsian.
\end{proof}

\vspace{2mm}

We remark  that Theorem 2.4 can be considered as a complex
hyperbolic analog  of a classical result due to B.Maskit, see
Theorem G.18 in \cite{Mas} and Corollary 3.2.5 in \cite{MaR}.

\vspace{2mm}

We say that $g \in {\rm PU}(2,1)$ is a screw motion iff any lift of
$g$ to ${\rm SU}(2,1)$ has non-real trace. For instance, purely
hyperbolic and unipotent parabolic elements are not screw motions.
Geometrically, $g \in {\rm PU}(2,1)$ is not screw motion iff  $g$
has an invariant  totally real geodesic 2-plane.

\begin{co}
Let $G$ be an irreducible discrete subgroup of  ${\rm PU}(2,1)$.
Then $G$ is an $\R$-subgroup if and only if $G$ contains no screw
motions.
\end{co}

\begin{co} Let $G$ be an irreducible discrete subgroup of ${\rm PU}(2,1)$
 whose limit set is not contained in an $\R$-circle. Then its invariant trace
 field is non-real extension of $\Q$.
\end{co}

\begin{co} Let $G$ be a discrete subgroup of ${\rm PU}(2,1)$ of
finite co-volume. Then its invariant trace field is non-real
extension of $\Q$.
\end{co}

\begin{proof}
It is clear that $G$ is irreducible. Suppose that the invariant
trace field $k(G)$ of $G$ is real. By Theorem 2.3, $G$ is an
$\R$-subgroup of ${\rm PU}(2,1)$. Therefore, the limit set of $G$ is
contained in an $\R$-circle. This implies that $G$ cannot have
finite co-volume.
\end{proof}

\vspace{3mm}

\noindent  \textbf{Acknowledgements.} We are very grateful to John
Parker for helpful conversations. We are also thankful to the
referee for useful remarks and suggestions.


\end{document}